\documentclass[12pt,reqno]{amsart}
\usepackage[top=1in, bottom=1in, left=1in, right=1in]{geometry}

\usepackage{times, amsthm, amssymb, amsmath, amsfonts, graphicx, mathrsfs}
\usepackage{mathtools}
\usepackage{tikz} 
\usetikzlibrary{arrows, cd, matrix,positioning}
\usepackage{xcolor}
\usepackage{bbm}
\usepackage{xypic}
\usepackage{mathscinet}

\newtheorem{theorem}{Theorem}[section]
\newtheorem{lemma}[theorem]{Lemma}
\newtheorem{proposition}[theorem]{Proposition}

\newtheorem{question}[theorem]{Question}

\newtheorem{introtheorem}{Theorem}

\theoremstyle{definition}
\newtheorem{definition}[theorem]{Definition}
\newtheorem{example}[theorem]{Example}
 
\theoremstyle{remark}
\newtheorem{remark}[theorem]{Remark}
\newtheorem*{fact}{Fact}

\newtheorem*{notation}{Notation}

\DeclareMathOperator{\Sym}{Sym}
\DeclareMathOperator{\rs}{rs}
\DeclareMathOperator{\RS}{RS}
\DeclareMathOperator{\coker}{coker}

\newcommand{\degree}{\operatorname{degree}}
\newcommand{\HF}{\operatorname{HF}}

\newcommand{\bbL}{\mathbbm{L}}

\newcommand{\bbR}{\mathbbm{R}}

\newcommand{\bbk}{\mathbbm{k}}

\newcommand{\frakm}{{\mathfrak{m}}}

\newcommand{\bfC}{{\mathbf{C}}}
\newcommand{\bfF}{{\mathbf{F}}}
\newcommand{\bfG}{{\mathbf{G}}}
\newcommand{\bfI}{{\mathbf{I}}}
\newcommand{\bfK}{{\mathbf{K}}}
\newcommand{\bfP}{{\mathbf{P}}}
\newcommand{\bfT}{{\mathbf{T}}}

\newcommand{\bfEN}{{\mathbf{EN}}}

\newcommand{\del}{{\partial}}

\renewcommand{\th}{{^\text{th}}}

\title{Subcomplexes of Certain Free Resolutions}
\author{Maya Banks}
\author{Aleksandra Sobieska}
\date{\today}

\thanks{The first author was partially supported by NSF grant DMS-1902123.}

\begin{document}

\maketitle

\begin{abstract}
What are the subcomplexes of a free resolution? This question is simple to state, but the naive approach leads to a computational quagmire that is infeasible even in small cases. In this paper, we invoke the Bernstein--Gel\cprime fand--Gel\cprime fand (BGG) correspondence to address this question for free resolutions given by two well-known complexes, the Koszul and the Eagon--Northcott. This novel approach provides a complete characterization of the ranks of free modules in a subcomplex in the Koszul case and imposes numerical restrictions in the Eagon--Northcott case.
\end{abstract}

\section{Introduction}

\begin{question}
What are the possible subcomplexes of free resolutions given by the Koszul or Eagon--Northcott complexes? 
\end{question}
The motivating question of this article is simple to state, but there are no traditional methods in commutative algebra to answer it, or indeed its more general sibling:

\begin{question}
Given a minimal free resolution $\bfG$ of a module over a local or graded ring, what are the subcomplexes $\bfF$ that have a split injective map into $\bfG$? 
\end{question}

We first learned of Question 1.1 for Koszul complexes in relation to ongoing work of Bertram and Ullery on Veronese secants and stable complexes. More broadly, we are also motivated by various uses of subcomplexes in the study of free resolutions. For instance, Question 1.2 is at the heart of recent work on virtual resolutions, where classifying and understanding subcomplexes with specific properties is the key to ~\cite[Theorem 3.1]{besVirtualResFAProdOProjSp} as well as related results such as ~\cite{harada2021virtual}. Even more generally, subcomplexes are fundamental objects of study in the world of syzygies; the linear strand, for example, plays an essential role in many results ~\cite{gKoszulCohomATGeomOProjVar, gKoszulCohomATGeomOProjVarII, ekSLinearSyzConj, gScrollarSyzOGenCanonicalCurvesWGenusle8}. 

The goal of this paper is to further understand the structure of subcomplexes from a numerical standpoint. This numerical approach fits in with the well-established broader approach to understanding minimal free resolutions numerically (for instance via the study of Betti tables or Poincaré series). What is more, the numerical realm is the natural place in which to explore our main question, since a change of basis introduces an infinite number of possible subcomplexes. 

In order to precisely state the numerical version of our question, we must first introduce some terminology. For a free complex $\bfF$, we define the \emph{rank sequence} of $\bfF$ to be the integer sequence $\rs(\bfF)~=~(r_0, r_1, \ldots )$ where $r_i$ is the rank of the $i^{th}$ free module in $\bfF$ (we refrain from calling these ``Betti numbers" since our subcomplexes may in general fail to be exact). For a complex $\bfG$, we will denote by $\RS(\bfG)$ the set of all integer sequences $r$ where $r = \rs(\bfF)$ for some subcomplex $\bfF$ of $\bfG$. Now we may ask the following:

\begin{question}
Given a free resolution $\bfG$, when is an integer sequence $r$ in $\RS(\bfG)$?
\end{question}

One might be tempted to approach the problem directly by trying to explicitly produce subcomplexes of free modules with prescribed ranks, but this raises certain subtleties even in small cases. We find that without some clear strategy for controlling subcomplexes, even this numerical question becomes hard. 

For a concrete example, let $S = \bbk[x_1, x_2, x_3]$ and $\bfG$ be the minimal free resolution of the residue field which is given by the Koszul complex on $x_1, x_2, x_3$, and consider the question ``Is there a subcomplex $\bfF$ of $
\bfG$ with ranks $r=(1,2,2,0)$?'' That is, ``Is $(1,2,2,0)$ in $\RS(\bfG)$?'' 
To answer this question directly requires a linear algebra analysis of the maps required to fill in a diagram like the one below, thus ultimately producing $\bfF$ in its entirety:

\[
\xymatrix{ 
    \bfF: 0 \ar[r] 
        & 0 \ar[r]\ar[d]
        & S(-2)^2 \ar[d]_{\varphi_{2}} \ar[r]^{f_2}
        & S(-1)^2 \ar[d]_{\varphi_1} \ar[r]^{f_1} 
        & S \ar[r] \ar[d]_{\varphi_0}
        & 0 \\
    \bfG: 0 \ar[r] 
        & S(-3) \ar[r]
        & S(-2)^3 \ar[r]
        & S(-1)^3 \ar[r]
        & S \ar[r] 
        & 0
    }
\]

Thus the restriction of the question to ranks does not give a commensurate improvement in obtaining an answer, and it is clear that a different approach is necessary.

This alternate approach uses the Bernstein--Gel'fand--Gel'fand (BGG) correspondence, which gives an equivalence of categories between graded linear complexes of free modules over a polynomial ring on the one hand and graded modules over an exterior algebra on the other. The question of subcomplexes thus becomes a question of submodules, where the restriction to possible ranks is translated to a question of possible Hilbert functions. Here we make use of existing results in a way that makes broad restrictions possible without constructing entire complexes.

In this paper, we provide some answers to the question of possible subcomplexes of resolutions for two classes of modules over the polynomial ring---complete intersections and quotients by some determinantal ideals---whose minimal free resolutions are given by the Koszul and Eagon--Northcott complexes respectively. The novelty of these theorems is twofold: besides providing constraints on permissible subcomplexes, they demonstrate the efficacy of the correspondence in tackling an otherwise intractable problem and provide insight on how similar results might be obtained for other free resolutions. In fact, since the BGG correspondence is an instance of Koszul duality, these techniques could extend to characterizing subcomplexes of linear complexes over general Koszul algebras, e.g. the Priddy complex.

Our first theorem exactly characterizes the integer sequences that can arise as ranks of subcomplexes of a minimal free resolution of a complete intersection. Let $S = \bbk[x_1, \ldots, x_n]$ and $\bfK(f_1, \ldots, f_m)$ be the Koszul complex on $f_1, \ldots, f_m$. 

\begin{introtheorem}[Theorem~\ref{thm:generalkoszulranks}]\label{thm:intro1}
Let $f_1, \ldots, f_m$ be homogeneous polynomials forming a regular sequence in the polynomial ring $S = \bbk[x_1, \ldots, x_n]$ and $r = (r_0, \ldots, r_m)$ be a non-negative integer sequence. Then $r\in \RS(\bfK(f_1, \ldots, f_m))$ if and only if  it is the zero sequence or $r_0 = 1$ and
\[
0\leq r_{i+1}\leq r_i^{(i)} for 1\leq i\leq n-1,
\]
where $r_i^{(i)}$ is the shifted Macaulay expansion of $r_i$ as in Definition~\ref{def:macaulayexpansion}.
\end{introtheorem}

We are then able to leverage this complete characterization along with the BGG correspondence to obtain meaningful restrictions on the integer sequences that can arise as ranks of subcomplexes of the Eagon--Northcott complex. Before we can state this result we need to introduce some more notation:
given complexes $\bfF$ and $\bfG$, we write $\RS(\bfF) + \RS(\bfG)$ for the set of sequences $r$ that can be written as an entry-wise sum $r = r_1 + r_2$ for $r_1\in \RS(\bfF)$ and $r_2\in \RS(\bfG)$. For an integer $a$, the set $a\RS(\bfF)$ is defined to be the $a$-fold sum $\RS(\bfG) + \cdots + \RS(\bfG)$.
We denote by $\bfK_{(m)}$ the Koszul complex on $m$ variables

With this notation, we can state our second main result about resolutions of ideals $I$ where the Eagon--Northcott complex gives a minimal free resolution of $S/I$.

\begin{introtheorem}[Theorem~\ref{thm:generalENranks}]\label{thm:intro2}
Let $S = \bbk[x_1, \ldots, x_n]$ and let $\phi$ be a $p\times q$ matrix with $p\leq q$ such that the maximal minors of $\phi$ generate an ideal $I$ of codimension $q-p+1$ where $S/I$ is Cohen-Macaulay. Let $\bfF$ be the minimal free resolution of $S/I$, and let $r$ be an integer sequence. If $r\in \RS(\bfF)$, then $r = (0, r_0, r_1,\ldots )$ or $r = (1, r_0, r_1, \ldots)$ where the sequence $(r_0, r_1, \ldots)$ is in 
\[
\sum_{j=0}^{q-p}\binom{q-j-1}{p-1}\RS(\bfK_{(q-p-j)}).
\]
\end{introtheorem}

\begin{example}
Let $S = \bbk[x_1,x_2,x_3]$, 
$\phi = \begin{bmatrix}
x_1 & x_2 & x_3 & 0 \\
0 & x_1 & x_2 & x_3
\end{bmatrix}$,
and $I$ be the ideal of $2 \times 2$ minors of $\phi$.
The minimal free resolution of $S/I$ is an Eagon--Northcott complex of the form
\[
\bfG: \quad 0 \to S(-4)^3 \to S(-3)^8 \to S(-2)^6 \to S,
\]
with maps as shown in Example~\ref{ex:eagonnorthcottforn3d2}.
We can find subcomplexes of $\bfG$ of the form
\[
0 \to S(-4)^2 \xrightarrow{f_3} S(-3)^6 \xrightarrow{f_2} S(-2)^5 \xrightarrow{f_1} S
\]

and
\[
0 \rightarrow S(-4)^1 \rightarrow S(-3)^3 \rightarrow S(-2)^3 \rightarrow S,
\]
but combining Theorem~\ref{thm:intro2} with the characterization of subcomplexes of the Koszul complex given in Theorem ~\ref{thm:intro1} rules out a subcomplex of $\bfG$ of the form
\[
0 \to S(-4)^3 \to S(-3)^5 \to S(-2)^5 \to S.
\]
This is because $(1,5,5,3)$ is not of the form $(1,r)$, where $r$ is in the set 
\[R = 3 \RS(\bfK_{(2)}) + 2 \RS(\bfK_{(1)}) + \RS(\bfK_{(0)}).\]
By Theorem~\ref{thm:intro1}, 
$\RS(\bfK_{(0)}) = \{(1,0,0), (0,0,0)\}$, 
$\RS(\bfK_{(1)}) = \{(1,1,0), (1,0,0), (0,0,0) \}$,
and 
$\RS(\bfK_{(2)}) = \{(1,2,1), (1,2,0), (1,1,0), (1,0,0), (0,0,0)\}$. If a sequence in $R$ has a $3$ in the last spot, we must use the sequence $(1,2,1)$ from $\RS(\bfK_{(2)})$ thrice. However, three times $(1,2,1)$ gives a $6$ in the middle position, so any sequence in $R$ with a $3$ in the last spot has at least a $6$ in the middle. Therefore $(5,5,3) \notin R$ and thus $(1,5,5,3) \notin \RS(\bfG)$. 
\end{example}

As seen in the example, the common theme throughout our results is that subtle numerics govern whether a given sequence of free modules and maps between them has any hope of being a complex. This situation is reminiscent of numerical conditions that tell us when a given complex can be exact -- a far more well-studied question. Many results are concerned with precisely characterizing exactness, while understanding when a sequence of maps is a complex is taken for granted. To riff on the title of \cite{beWhatMakesAComplexExact}, herein we set out to explore the dual question: ``What makes a complex a complex?''

\subsection*{Acknowledgments} The authors would like to thank Daniel Erman for originally suggesting the motivating question and many helpful discussions throughout all stages of the article's development. We also thank the anonymous referees for their thoughtful comments and suggestions.

\section{Background}\label{section:background}

For the sake of clarity, we settle a formal definition of ``subcomplex.'' 

\begin{definition}
Let $\bfF=(F_i,f_i)$ and $\bfG=(G_i,g_i)$ be two complexes of free modules over the same ring. We say $\bfF$ is a \textit{subcomplex} of $\bfG$ if there are split injective maps $\varphi_i: F_i \rightarrow G_i$ so that $\varphi_i \circ f_{i+1} = g_{i+1} \circ \varphi_{i+1}$, i.e. each square of the following diagram commutes:
\[
\xymatrix{ 
    \bfF: \cdots \ar[r] & F_{i+1} \ar[d]_{\varphi_{i+1}} \ar[r]^{f_{i+1}} & F_{i} \ar[d]_{\varphi_{i}} \ar[r]^{f_{i}} & F_{i-1} \ar[d]_{\varphi_{i-1}} \ar[r]^{f_{i-1}} & \cdots \\
    \bfG: \cdots \ar[r] & G_{i+1} \ar[r]_{g_{i+1}} & G_{i} \ar[r]_{g_i} & G_{i-1} \ar[r]_{g_{i-1}} & \cdots 
    }
\]
\end{definition}

In particular, we exclude injective maps like $\varphi_i: G_i(-1) \xrightarrow{\cdot x} G_i$. In the cases we are interested in, these $\varphi_i$ can be represented by matrices with full column rank and entries from the ground field $\bbk$. 

Given a free complex $\bfG$, our goal will be to classify the ranks of free modules appearing in subcomplexes of $\bfG$. We introduce some notation that will be used throughout.

\begin{definition}
Given a free complex $\bfF = \cdots \to F_1 \to F_0 $, the \emph{rank sequence} of $\bfF$ is 
\[\rs(\bfF) = (r_0, r_1, \ldots)\] 
where $r_i$ is the rank of the free module $F_i$. For a complex $\bfG$, we use $\RS(\bfG)$ to denote the set of all possible rank sequences of subcomplexes of $\bfG$.
\end{definition}

\begin{notation} Given two sets of rank sequences, say $A = \RS(\bfF)$ and $B = \RS(\bfG)$, we will write $A+B$ to refer to the set of sequences that may be expressed as a sum of a sequence in $A$ and a sequence in $B$. Similarly, we will write $nA$ to refer to the set $A+A+\cdots + A$ where the sum has $n$ terms.
\end{notation}

\subsection{The BGG Correspondence}

The key tool for our results is the Bernstein-Gel'fand-Gel'fand correspondence \cite{bggAlgVecBunOPnAProbOLinAlg}, which allows us to translate questions about linear free complexes of modules over a symmetric  algebra into questions about modules over an exterior algebra. We will cherry-pick what we need of this rich subject; for further detail, see \cite{efsSheafCohomAFreeResOExtAlg} and \cite[Section~7B]{eGeomOSyz}.

Let $\bbk$ be a field, $V$ be a $\bbk$-vector space with basis $x_1, \ldots, x_n$, and $W$ be the dual vector space of $V$ with basis $e_1, \ldots, e_n$. Let $E = \bbk\langle e_1, \ldots, e_n\rangle$ denote the exterior algebra on $W$.  Let $S$ denote the symmetric algebra $\Sym(V)$, and identify $S$ with the polynomial ring $\bbk[x_1, \ldots, x_n]$. We will assume that the $x_i$ are graded in degree $1$, and the $e_i$ are graded in degree $-1$. Unless otherwise stated, all tensor products are assumed to be over the ground field $\bbk$.

We define a pair of functors $\bbL$ and $\bbR$ as follows:
\begin{align*}
\bbL: \{\text{Graded $E$-modules}\} &\to \{\text{Linear complexes of free $S$-modules}\}\\
N &\longmapsto (\cdots\to S\otimes N_d\xrightarrow[]{\partial_d} S\otimes N_{d-1}\to\cdots)
\end{align*}
with differential $\partial_d$ defined by linearly extending
\[
1\otimes f\mapsto \sum_{i=1}^n x_i\otimes fe_i
\]
and
\begin{align*}
\bbR: \{\text{Graded $S$-modules}\} &\to \{\text{Linear complexes of free $E$-modules}\}\\
M &\longmapsto (\cdots\to E\otimes M_d\xrightarrow[]{\partial_d} E\otimes M_{d+1}\to\cdots)
\end{align*}

with differential $\partial_d$ defined by linearly extending
\[
1\otimes g\mapsto \sum_{i=1}^n e_i\otimes gx_i.
\]

One can check that the functors $\bbL$ and $\bbR$ preserve exactness.

\begin{example}\label{ex:1441}
Consider the module $N = \langle e_1, e_2e_3 \rangle E$ where $E = \bbk\langle e_1, e_2, e_3, e_4\rangle$. We will use the following $\bbk$-bases for the graded pieces of $N$: 
\[
\begin{array}{ll}
    \degree \ -1 &: e_1 \\
    \degree \ -2 &: e_1e_2, \ e_1e_3, \ e_1e_4, \ e_2e_3 \\ 
    \degree \ -3 &: e_1e_2e_3, \ e_1e_2e_4, \ e_1e_3e_4, \ e_2e_3e_4 \\
    \degree \ -4 &: e_1e_2e_3e_4
\end{array}
\]

Tracing through the definition of $\bbL$ we can see, for example, that \[\del_{-2}(1 \otimes e_1e_2) = \sum\limits_{i=1}^4 x_i \otimes e_1e_2e_i = x_3 \otimes e_1e_2e_3 + x_4 \otimes e_1e_2e_4.\] 

The entirety of $\bbL(N)$ is the complex: 
\[ 
0 \rightarrow S \otimes N_{-1}
\xrightarrow{
\begin{bmatrix}
x_2 \\ x_3 \\ x_4 \\ 0
\end{bmatrix}
}
S \otimes N_{-2}
\xrightarrow{ 
\begin{bmatrix}
x_3 & -x_2 & 0 & x_1 \\ 
x_4 & 0 & -x_2 & 0 \\ 
0 & x_4 & -x_3 & 0 \\ 
0 & 0 & 0 & x_4
\end{bmatrix}
}
S \otimes N_{-3}
\xrightarrow{
\begin{bmatrix}
x_4 & -x_3 & x_2 & -x_1
\end{bmatrix}
}
S \otimes N_{-4} \rightarrow 0
\]
\end{example}

The BGG correspondence states that if we consider $\bbL$ and $\bbR$ as functors on the bounded derived categories then they are adjoint, implying that the derived categories of bounded linear complexes of finitely generated graded $E$-modules and $S$-modules are equivalent. What is more, the functor $\bbL$ gives a bijection on objects under which 
\begin{enumerate}
    \item Any linear complex $\bfF$ of $S$-modules may be expressed as $\bbL(N)$ for some $E$-module $N$ \cite{efsSheafCohomAFreeResOExtAlg}, and
    \item Subcomplexes of $\bfF$ correspond to $E$-submodules of $N$.
\end{enumerate}

\begin{remark}\label{rmk:hfrs} For $\bfF = \bbL(N)$, we can therefore relate the Hilbert function of $N$ and the rank sequence $\rs(\bfF)$. Take $r = (r_0, r_1, \ldots, r_n)$ and $h=(h_0, h_1, \ldots, h_n)$ to be two sequences of non-negative integers. Then $r = \rs(\bfF)$ if and only if $h_i = r_{n-i}$ is the Hilbert function of $N$, that is, if $h_i = \dim_\bbk(N_{-i}) = r_{n-i}$. Note that we are still considering the Hilbert function $h$ as a function from $\mathbb{N}$ to $\mathbb{N}$, despite the negative grading on $E$. We will occasionally commit the minor sin of conflating $h$ as a function and an integer sequence, and thus write $h(N)$ for the sequence $(h_0, h_1, \ldots, h_n)$ where $h_i = \dim_\bbk(N_{-i})$. 
\end{remark}

\begin{remark}
A quick check reveals that shifting the homological degree of a complex $\bfF$ corresponds with twisting an $E$-module by that same degree, that is, if $\bbL(N) = \bfF$, then $\bbL(N(i)) = \bfF[i]$, where $\bfF[i]_j = \bfF_{i+j}$.
\end{remark}

We also make use of the following relationship between $\bbL$ and $\bbR$:

\begin{theorem}{(Reciprocity Theorem) \cite[Theorem~3.7]{efsSheafCohomAFreeResOExtAlg}} \label{thm:reciprocitythm}
Let $M$ be a graded $S$-module and let $N$ be a graded $E$-module. Then 
\[
N\to \bbR(M)
\]
is an injective resolution if and only if 
\[
\bbL(N)\to M
\]
is a free resolution.
\end{theorem}

\subsection{Tate Resolutions}

\begin{definition}
For any module $N$ over any ring, we can combine a projective resolution $\bfP$ of $N$ and an injective resolution $\bfI$ of $N$ in the following way 
\[
\begin{tikzcd}
\cdots \arrow[r, "\del_2"] & P_1 \arrow[r, "\del_1"] & P_0 \arrow[rd] \arrow[rr, "\del_0"] & 
& I_0 \arrow[r, "\del_{-1}"] & I_1 \arrow[r, "\del_{-2}"] & \cdots \\
& & & N \arrow[rd] \arrow[ru] & & & \\
& & 0 \arrow[ru] & & 0 & & 
\end{tikzcd}
\]
to produce a \textit{Tate resolution}. 
\end{definition}

More detail about general Tate resolutions can be found in \cite{efsSheafCohomAFreeResOExtAlg}, but 
we are most interested in Tate resolutions of modules over $E$, where injective and projective modules are both free. In this case, we can take $\bfP$ to be a minimal free resolution of $N$ and $\bfI$ to be the dual of the minimal free resolution of the dual of $N$ to create a unique doubly infinite exact complex of free modules where the image of $P_0$ is isomorphic to $N$. We will call this doubly infinite complex \textit{the} Tate resolution $\bfT(N)$. 

\begin{example}\label{ex:tateres}
If $N = E/\langle e_1, \ldots, e_n \rangle \cong \bbk$, then the Cartan resolution $(\bfC,\del)$ is a projective resolution of $N$ (cf. \cite[Corollary~7.10]{eGeomOSyz}). The dual of $\bfC$ is an injective resolution of $\bbk$ (which is its own dual), so stitching the two together yields the Tate resolution of $\bbk$. Below is a snippet of $\bfT(\bbk)$ in the $n=3$ case: 
\[ 
\hspace{-1cm}
\begin{tikzcd}
\cdots \arrow[r] & E^6(2) \arrow[r, "\del_2"] & E^3(1) \arrow[r, "\del_1"] & E \arrow[rd] \arrow[rr, "\del_0"] & 
& E(-3) \arrow[r, "\del_1^T"] & E^3(-4) \arrow[r, "\del_2^T"] & E^6(-5) \arrow[r] & \cdots \\
& & & & \bbk \arrow[rd] \arrow[ru] & & & & \\
& & & 0 \arrow[ru] & & 0 & & & 
\end{tikzcd}
\] 
where  
$
\del_0 = \begin{bmatrix} 
e_1 e_2 e_3
\end{bmatrix}$, 
$\del_1 = \begin{bmatrix}
e_1 & e_2 & e_3
\end{bmatrix}$,
and 
$\del_2 = \begin{bmatrix}
e_1 & e_2 & 0 & e_3 & 0 & 0 \\
0 & e_1 & e_2 & 0 & e_3 & 0 \\
0 & 0 & 0 & e_1 & e_2 & e_3
\end{bmatrix}.
$

In general, the differential $\del_s$ in the Cartan resolution can be computed by indexing the columns of $\del_s$ with the degree-$s$ monomials in the $x_i$'s and the rows by the degree $s-1$ monomials in the $x_i$'s. Then, if column $i$ is indexed by a monomial $m$ and row $j$ is indexed by a monomial $m'$, the $(i,j)$th entry of $\del_s$ is $e_k$ if $m/m' = x_k$ if $m' \mid m$ and $0$ if $m \nmid m'$. In Example~\ref{ex:tateres}, the indexing monomials for the entries of the $\del_s$ are listed in graded reverse lexicographic order with $x_1 > x_2 > x_3$. 

\begin{remark} \label{rem:tateresformoreambient} 
Because $E$ is free over $\bbk\langle e_1, \ldots, e_m \rangle$ for $m \leq n$, extending scalars from $\bbk\langle e_1, \ldots, e_m \rangle$ to $E$ is faithfully flat. This means the Tate resolution $\bfT(E/\langle e_1, \ldots, e_m \rangle)$ has the same structure as the Tate resolution $\bfT(\bbk\langle e_1, \ldots, e_m\rangle/\langle e_1, \ldots, e_m \rangle)$ as a complex of $\bbk\langle e_1, \ldots, e_m\rangle$-modules. That is, the complex of $\bbk\langle e_1, \ldots, e_m\rangle$-modules $\bfT(\bbk)$ and the complex of $E$-modules $\bfT(\bbk\langle e_{m+1}, \ldots, e_n \rangle)$ have modules of the same rank and twists, and differentials with the same entries, regardless of the ambient ring. For example, the Tate resolution $\bfT(\bbk\langle e_4\rangle)$ over $\bbk \langle e_1, \ldots, e_4\rangle$ will ``look'' the same as the one shown in Example~\ref{ex:tateres}, with all $E$'s replaced by $\bbk\langle e_1, \ldots, e_4 \rangle$.
\end{remark}

\end{example}

\section{Resolutions of $\frakm^d$}

As before, let $S = \bbk[x_1, \ldots, x_n]$, where $\bbk$ is a field. Use $\frakm$ to denote the homogeneous maximal ideal $\langle x_1, \ldots, x_n \rangle$. We begin by exploring the possible rank sequences of subcomplexes of resolutions of $\frakm^d$, in particular as they are presented by the Koszul complex in the $d=1$ case and the Eagon--Northcott complex in the $d\geq 2$ case.

\subsection{The Koszul Complex}

\begin{definition}\label{def:koszul} 
The Koszul complex $\bfK(x_1, \ldots, x_m)$ is the graded exact complex 
\[ 
\bfK(x_1, \ldots, x_m): 0 \rightarrow S(-m) \xrightarrow{\del_m} \cdots \xrightarrow{\del_3} S(-2)^{\binom{m}{2}} \xrightarrow{\del_2} S(-1)^m \xrightarrow{\del_1} S^1 \rightarrow 0,
\] 
where we index basis elements of $K_d \coloneqq S(-d)^{\binom{m}{d}}$ by size $d$ subsets of $m$. For $T~=~\{i_1, \ldots, i_d\}$, the differential $\del_d$ acts on $e_T$ by $\del_d(e_T) = \sum_{j = 1}^{d} (-1)^{j} x_{i_j} e_{T \backslash i_j}$. 
\end{definition} 

\begin{example}\label{ex:koszul3}
The Koszul complex $\bfK(x_1,x_2,x_3)$ is given by 
\[
\bfK(x_1, x_2, x_3): 0 \rightarrow S(-3) \xrightarrow{\del_3} S(-2)^3 \xrightarrow{\del_2} S(-1)^3 \xrightarrow{\del_1} S \rightarrow 0
\]
where 
\[ 
\del_1 = 
\begin{tikzpicture}[anchor=base, baseline]
    \matrix [matrix of math nodes,left delimiter={[},right delimiter={]}] (d1)
        {x_1 & x_2 & x_3 \\};
    \draw [color=red] (d1-1-1.north west) rectangle (d1-1-2.south east);
\end{tikzpicture} 
\quad 
\del_2 = 
\begin{tikzpicture}[baseline]
    \matrix [matrix of math nodes,left delimiter={[},right delimiter={]}] (d2)
        {-x_2 & -x_3 & 0 \\ 
        x_1 & 0 & -x_3 \\ 
        0 & x_1 & x_2 \\};
    \draw [color=red] (d2-1-1.north west) rectangle (d2-2-1.south east);
\end{tikzpicture} 
\quad
\text{ and } 
\quad 
\del_3 = 
\begin{tikzpicture}[baseline]
    \matrix [matrix of math nodes,left delimiter={[},right delimiter={]}] (d3)
        {x_3 \\ -x_2 \\ x_1 \\};
\end{tikzpicture}. 
\]
\end{example}

Looking at the above example, we can immediately identify some subcomplexes of the Koszul complex. If $m \leq n$, the complex $\bfK(x_1, \ldots, x_m)$ is a subcomplex of $\bfK(x_1, \ldots, x_n)$ -- one can see the Koszul complex $\bfK(x_1,x_2)$ boxed in red in Example~\ref{ex:koszul3}. One can also simply truncate $\bfK(x_1, x_2, x_3)$ after two modules and omit the $S(-3)$ module at the end, or even omit $S(-3)$ and some summands of the $S(-2)^3$ in the next spot. This observation yields certain sequences that we can be sure must occur as rank sequences of subcomplexes of $\bfK(x_1, \ldots, x_m)$ but to obtain a more complete classification we can peer through the BGG lens and, in particular, use the following key fact.

\begin{fact}[see Example~7.6, \cite{eGeomOSyz}]
The linear complex $\bbL(E(-n))$ is (isomorphic to) the Koszul complex $\bfK(x_1, \ldots, x_n)$.
\end{fact}

The BGG correspondence thus tells us that subcomplexes of the Koszul complex are in correspondence with submodules of the exterior algebra $E$ twisted by $(-n)$, so our question about the possible rank sequences of subcomplexes of $\bfK$ is transformed into a question about the possible Hilbert functions of submodules of $E$ itself (after the appropriate twist). This perspective immediately reveals that subcomplexes of $\bfK$ are less restricted than one might guess from the $n=3$ case. Indeed, Example~\ref{ex:1441} shows that we can obtain a subcomplex of $\bfK(x_1, \ldots, x_4)$ whose rank sequence is $(1,4,4,1,0)$, which is not the rank sequence of a smaller Koszul complex and furthermore cannot be obtained by truncating free summands from the tail of $K$.

This observation also underscores the complexity of the structural question of classifying all subcomplexes in the case of the Koszul complex. Such a task would be equivalent to classifying all ideals in $E$. Though the feasibility of such classification is yet unknown, it is worth noting that the parallel question of classifying ideals in $S$ is impossible by Vakil's Murphy's Law \cite{vMurphysLawIAlgGeoBadBehavedDefSpaces}.

By work of Aramova--Herzog--Hibi \cite[Theorem~4.1]{ahhGotzmannThFExteriorAlgACombinatorics}, possible Hilbert sequences of submodules of the exterior algebra are exactly those corresponding to $f$-vectors of simplicial complexes as described by the Kruskal--Katona Theorem. We can use these results to characterize the possible rank sequences for a subcomplex of the Koszul complex with the following notation.

\begin{definition}\label{def:macaulayexpansion}
If $a$ is a positive integer, then, for every positive integer $i$, $a$ has a unique \emph{Macaulay expansion}
\[
a = \binom{a_i}{i} + \binom{a_{i-1}}{i-1}+\cdots +\binom{a_j}{j},
\]
where $a_i>a_{i-1}>\cdots>a_j\geq j\geq 1$. Define 
\[
a^{(i)}:= \binom{a_i}{i+1} + \binom{a_{i-1}}{i}+\cdots +\binom{a_j}{j+1}.
\]
\end{definition}

\begin{theorem}\label{thm:basickoszulranks}
A non-negative integer sequence $(r_0, r_1, \ldots, r_n)$ is in $\RS(\bfK(x_1, \ldots, x_n))$ if and only if it is the all zeroes sequence or if $r_0=1$ and $r$ satisfies 
\[
0\leq r_{i+1}\leq r_i^{(i)} \text{ for } 1 \leq i \leq n-1.
\]
\end{theorem}

\begin{proof}
If $r$ is the sequence of all zeroes, it is the rank sequence of the zero complex, which is a subcomplex of any complex.

Using Corollary~5.3 from \cite{ahhGotzmannThFExteriorAlgACombinatorics} and Remark~\ref{rmk:hfrs}, we see that $h(E/(0:I)) = \rs(\bbL(I))$. Because every ideal in $E$ satisfies $0:(0:I) = I$ and can therefore be recognized as an annihilator, classifying Hilbert sequences $h(E/I)$ is equivalent to classifying Hilbert sequences $h(E/(0:I))$. Combining these two sentences, we see that classifying rank sequences $\rs(\bbL(I))$ is equivalent to classifying rank sequences $h(E/I)$. 

By \cite[Theorem~4.1]{ahhGotzmannThFExteriorAlgACombinatorics}, a non-negative integer sequence $h=(1,h_1, \ldots, h_n)$ is the Hilbert sequence of a module $E/I$ if and only if $0 \leq h_{i+1} \leq h_i^{(i)}$ for all $1 \leq i \leq n-1$. This translates directly to the set $\RS(\bfK(x_1, \ldots, x_n))$, and our theorem is proven.
\end{proof}

\begin{remark}
In an analogous way, Macaulay's Theorem (c.f. \cite[Theorem~4.2.14]{bhCMRings}) characterizes the ranks of subcomplexes of the Cartan resolution of $\bbk$ over $E$.
\end{remark}

\subsection{The Eagon--Northcott Complex}
The Eagon--Northcott complex \cite{enIdDefBMatAACComplexAssocWThem} 
plays the same role for determinantal ideals that a Koszul complex plays for a sequence of ring elements. 
We provide a brief presentation here that describes the complex for $S$-modules; more details can be found in \cite[Appendix~A2H]{eGeomOSyz}. Throughout, we will choose bases for our free modules so that we can represent these maps as matrices. 

\begin{definition}
Let $F = S^f$ and $G=S^g$, with $g \leq f$, and $\alpha: F \rightarrow G$ a map represented by a $g \times f$ matrix $A$ with respect to bases $\{e_1, \ldots, e_f\}$ of $F$ and $\{\varepsilon_1, \ldots, \varepsilon_g\}$ of $G$. Then the Eagon--Northcott complex of the map $\alpha$ is the complex
\[ 
\bfEN(\alpha): 0 \rightarrow EN_{f-g+1} \xrightarrow{d_{f-g+1}} EN_{f-g} \xrightarrow{d_{f-g}} \cdots \xrightarrow{d_3} EN_2 \xrightarrow{d_2} EN_1 \xrightarrow{\Lambda^g \alpha} \Lambda^g G
\]
where $EN_{k+1} = (\Sym_k G)^* \otimes \Lambda^{g+k} F$ 
and $d_{k+1}: (\Sym_k G)^* \otimes \Lambda^{g+k} F \rightarrow (\Sym_{k-1} G)^* \otimes \Lambda^{g+k-1} F$ is the map
\begin{multline*} 
(\varepsilon_{1}^{p_1} \cdots \varepsilon_{g}^{p_g})^* \otimes e_{s_1} \wedge \cdots \wedge e_{s_{g+k}}
\mapsto \\
\sum\limits_{i=1}^{g+k} (-1)^{i-1}
\left[ \sum\limits_{j=1}^{g} A_{j,s_i} (\varepsilon_1^{p_1} \cdots \varepsilon_j^{p_j-1} \cdots \varepsilon_g^{p_g})^* \right] 
\otimes e_{s_1} \wedge \cdots \wedge \widehat{e_{s_i}} \wedge \cdots \wedge e_{s_{g+k}}.
\end{multline*}
for $k \geq 1$, where $p_1 + \ldots + p_g = k$ and we adopt the convention that $\varepsilon_j^p = 0$ if $p < 0$.
\end{definition}

Note that, if we represent $\alpha$ by the matrix $A$, then $A_{j,s} = \varepsilon_j^*(\alpha(e_s))$, and that using a different basis to express $A$ will give an isomorphic complex.

\begin{example}\label{ex:eagonnorthcottforn3d2}
For this example, let $S= \bbk[x_1, x_2, x_3]$. Consider the example $\alpha:~S^4~\rightarrow~S^2$ represented by the matrix 
$A = \begin{bmatrix} 
x_1 & x_2 & x_3 & 0 \\
0 & x_1 & x_2 & x_3 \\
\end{bmatrix}.$

Making the appropriate identifications for each module, we can see that 
$\bfEN(\alpha)$ is the resulting graded complex 
\[
\bfEN(\alpha): 0 \rightarrow S(-4)^3 \xrightarrow{d_3} S(-3)^8 \xrightarrow{d_2} S(-2)^6 \xrightarrow{d_1} S.
\]

Note that the ideal of maximal minors in Example~\ref{ex:eagonnorthcottforn3d2} is the ideal $\langle x_1, x_2, x_3 \rangle^2$. In general, the Eagon--Northcott complex minimally resolves any power of the maximal ideal $\frakm^d$ by constructing the complex for the $d \times (n+d-1)$ matrix 
\[
M^{n,d} = \begin{bmatrix}
x_1 & x_2 & \ldots & x_n & 0 & \ldots & \ldots & 0 \\ 
0 & x_1 & x_2 & \ldots & x_n & 0 & \ldots & 0 \\ 
\vdots & \ddots & \ddots & \ddots & & \ddots & \ddots & \vdots \\
0 & \ldots & 0 & x_1 & x_2 & \ldots & x_n & 0 \\ 
0 & \ldots & \ldots & 0 & x_1 & x_2 & \ldots & x_n \\
\end{bmatrix}.
\]

\end{example}

With this notation, the matrix $A$ in Example~\ref{ex:eagonnorthcottforn3d2} is $M^{3,2}$. 

Our eventual goal is to understand, up to rank sequence, the possible subcomplexes of the Eagon--Northcott complex in a similar way as we did with the Koszul complex. We will see that these sequences are much more difficult to classify completely, but that we can still find restrictions that narrow down the possibilities. To do this, we will once again leverage the BGG correspondence in order to understand subcomplexes of Eagon--Northcott complexes. In order to do this, we must restrict ourselves to the linear maps -- and therefore the degree $d$ strand -- of the Eagon--Northcott complex. Note however that with the exception of the first map, the Eagon--Northcott complex is always linear, and any subcomplex of the degree $d$ strand will extend to a subcomplex of the entire Eagon--Northcott complex. 

We define the complex $L_{n,d}$ to be the resolution of the ideal $\langle x_1, \ldots, x_n \rangle^d$ as an $S$-module as presented by $\bfEN(M^{n,d})$ and note that $L_{n,d}$ is linear. Indeed, for $d\geq 2$, $L_{n,d}$ is the degree--$d$ strand of $\bfEN(M^{n,d})$ shifted by one homological degree, while for $d=1$ the entire complex is linear, so $L_{n,1}$ is not the entire linear strand since it is missing the first map. This complex corresponds to a specific $E$-module $N_{n,d}$ under the BGG correspondence, that is, $L_{n,d} =\mathbb{L}(N_{n,d})[-n+1]$. Therefore, classifying subcomplexes of $L_{n,d}$ corresponds to understanding the $E$-submodules of $N_{n,d}(-n+1)$. 

\begin{remark}\label{rmk:linearstrand}
Given an subcomplex $\bfF$ of $L_{n,d}$, we can always extend to a subcomplex of $\bfEN(M^{n,d})$. In fact, since the first term of $\bfEN(M^{n,d})$ is $S^1$, we can extend $\bfF$ by either $0$ or by $S^1$ to obtain a subcomplex of $\bfEN(M^{n,d})$. At the level of rank sequences, this means that $r\in \RS(\bfEN(M^{n,d}))$ has the form $(0,r')$ or $(1,r')$ for $r'\in \RS(L_{n,d})$.
\end{remark}

\begin{proposition}
The module $N_{n,d}$ is the cokernel of $\del_{d-1}^T$ in the Tate resolution $\bfT(\bbk)$. 
\end{proposition}

\begin{proof}
To obtain a presentation for $N_{n,d}$, we can appeal to the Reciprocity Theorem (Theorem~\ref{thm:reciprocitythm}) and the Tate resolution $\bfT(\bbk)$. Because $\bbL(N_{n,d}) \rightarrow \frakm^d$ is a free resolution, $N_{n,d} \rightarrow \bbR(\frakm^d)$ is an injective resolution, so $N_{n,d}$ is the kernel of $\bbR(\frakm^d)$. 

The complex $\bbR(\frakm^d)$ is known and, in fact, fits nicely into the Tate complex $\bfT(\bbk)$: if $\bbk \rightarrow I_0 \xrightarrow{\del_1^T} I_1 \xrightarrow{\del_2^T} \cdots$ is the injective resolution of $\bbk$ as an $E$-module,
then $\bbR(\frakm^d)$ is the truncation $0 \rightarrow I_d \xrightarrow{\del_{d+1}^T} I_{d+1} \xrightarrow{\del_{d+2}^T} \cdots$. This means that $N_{n,d} = \ker \del_{d+1}^T$, which we can rewrite by the Tate resolution as $\coker \del_{d-1}^T$, since the Tate resolution is the unique exact way to extend $\bbR(\frakm^d)$ to the left. 
\end{proof}

\begin{example}
We expound upon the case shown in Example~\ref{ex:tateres}, which considers the matrix $M^{3,2}$ whose minors give the ideal $\langle x_1, x_2, x_3 \rangle^2 \subseteq S = \bbk[x_1, x_2, x_3]$. The complex $L_{3,2}$ corresponds to the $E$-module $N_{3,2} = \coker \del_1^T$, where $\del_1^T$ is the map in the Tate resolution given in Example~\ref{ex:tateres}. 
\end{example}

In this way, our search for possible rank sequences of subcomplexes of $L_{n,d}$ is translated to a search for possible Hilbert functions of submodules of $N_{n,d}$ (again with appropriate twist). We denote this set of the possible Hilbert functions of submodules of $N_{n,d}$ by $\HF(N_{n,d})$ and prove this set satisfies certain constraints. Some persnickety bookkeeping is necessary proceeding. 

\begin{remark}
The module $N_{m,d}$ has the same presentation matrix when viewed as a $\bbk \langle e_1, \ldots, e_m \rangle$-module and as a $\bbk \langle e_1, \ldots, e_n \rangle$-module. This follows from combining the argument for the presentation of $N_{n,d}$ and Remark~\ref{rem:tateresformoreambient}. However, the Hilbert function is \textit{not} the same when we consider $N_{m,d}$ as a $\bbk \langle e_1, \ldots, e_m \rangle$-module and as a $\bbk \langle e_1, \ldots, e_n \rangle$-module. For example, as a $\bbk\langle e_1, e_2 \rangle$-module, $N_{1,1} = \coker \begin{bmatrix} e_1 \end{bmatrix}$ has Hilbert function $(1,1,0)$. This differs from the Hilbert function of $\coker \begin{bmatrix} e_1 \end{bmatrix}$ when viewed as a $\bbk\langle e_1 \rangle$-module, which is simply $(1,0,0)$. Therefore, in general, the set $\HF(N_{m,d})$ will vary, depending on the ambient exterior algebra, so we must introduce more precise notation. We will continue to let $E = \bbk\langle e_1, \ldots, e_n\rangle$ and $N_{n,d}$ for the module where $L_{n,d} = \bbL(N_{n,d}(-n+1))$. If we are considering the module $N_{m,d}$ as an $E$-module, we will use the notation $\overline{N_{m,d}}$. Note that $\overline{N_{m,d}} = N_{m,d} \otimes \bbk\langle e_{m+1}, \ldots, e_{n} \rangle$. 
\end{remark}

\begin{theorem}\label{thm:HFNndcontainment}
The set of possible Hilbert functions of submodules of $N_{n,d}$ is restricted by the following containment: 
\[\HF(N_{n,d}) \subseteq \HF(N_{n,d-1}) + \HF(\overline{N_{n-1,d}}). \]
\end{theorem}

\begin{proof}
First we prove that 
\[ 
0 \rightarrow N_{n,d-1} \rightarrow N_{n,d} \rightarrow \overline{N_{n-1,d}} \rightarrow 0 
\]
is a short exact sequence. Note that $\frakm^d$ is the $S$-module $S_{\geq d}$, from which we get the exact sequence of $S$-modules
\[ 
0 \rightarrow S_{\geq d-1}(-1) \rightarrow S_{\geq d} \rightarrow S'_{\geq d} \rightarrow 0,
\]
where $S' = \bbk[x_1, \ldots, x_{n-1}]$ is an $S$-module in the usual way: $x_n \cdot f = 0$ for any $f \in S'$. 

Because $\bbL$ preserves exactness, this means that 
\[
0 \rightarrow \bbR(S_{\geq d-1}(-1)) \rightarrow \bbR(S_{\geq d}) \rightarrow \bbR(S'_{\geq d}) \rightarrow 0 
\]
is a short exact complex of linear complexes of $E$-modules. In particular, we know what the kernel of each complex is: exactly the corresponding module $N$. Therefore 
\[
0 \rightarrow N_{n,d-1} \rightarrow N_{n,d} \rightarrow \overline{N_{n-1,d}} \rightarrow 0
\]
is indeed a short exact sequence of $E$-modules. 

Now suppose that $N \subseteq N_{n,d}$ is a submodule with Hilbert function $h(N)$. The image of $N$ in $\overline{N_{n-1,d}}$ is also a submodule, which we will denote $N'$. Let $N''$ be the kernel of the induced map $N\to N'$. It is a submodule of $N_{n,d-1}$, so we have a short exact sequence 
\[
0\to N''\to N \to N' \to 0.
\]

Because Hilbert functions sum over short exact sequences, we have $h(N) = h(N') + h(N'')$, so the Hilbert function of $N\subseteq N_{n,d}$ is realized as a sum of Hilbert functions of submodules of $\overline{N_{n-1,d}}$ and $N_{n,d-1}$. Thus we see that 
\[
\HF(N_{n,d})\subset \HF(\overline{N_{n-1,d}}) + \HF(N_{n,d-1}).
\]
\end{proof}
Note that the containment in Theorem~ \ref{thm:HFNndcontainment} is not an equality. We have shown that any Hilbert function of a submodule of $N_{n,d}$ may be realized as a sum of Hilbert functions of submodules of $\overline{N_{n-1,d}}$ and $N_{n,d-1}$, but there may be submodules of $\overline{N_{n-1,d}}$ and $N_{n, d-1}$ the sum of whose Hilbert functions is not the Hilbert function of a submodule of $N_{n,d}$. Indeed, the following example shows that the containment is strict in even a very small case.

\begin{example}\label{ex:N21}
For this example, we will use $E = \bbk\langle e_1, e_2 \rangle$ and consider the $E$-module $N_{2,2}$. From above, we have 
\[
\HF(N_{2,2}) \subseteq \HF(\overline{N_{1,2}}) + \HF(N_{2,1})
\]

Recall that $N_{2,2} = \coker\begin{bmatrix} e_1\\e_2\end{bmatrix}$, $\overline{N_{1,2}} = \coker\begin{bmatrix} e_1\end{bmatrix}$, and $N_{2,1} = \coker\begin{bmatrix} e_1e_2\end{bmatrix}$. The module $\overline{N_{1,2}}$ has Hilbert function $(1,1,0)$, while the submodule $0 \subset N_{2,1}$ has Hilbert function $(0,0,0)$, so we get the sum $(1,1,0)+(0,0,0)$ as a potential Hilbert function for a submodule of $N_{2,2}$. However, there is no submodule of $N_{2,2}$ with Hilbert function $(1,1,0)$ by the following argument. 

The module $N_{2,2}$ has $2$ generators in degree $0$, which we will call $\alpha$ and $\beta$. In degree $-1$ we have $e_1\alpha, e_1\beta, $and $e_2\alpha$, with $e_1\alpha = -e_2\beta$. Suppose we have a submodule $N\subset N_{2,1}$ with one generator in degree $0$. If we denote this degree $0$ generator of $N$ by $\zeta$, then we have that $\zeta = a\alpha + b\beta$ for some $a,b\in \bbk$. This gives us $2$ elements in degree $-1$: $e_1\zeta = ae_1\alpha + be_2\beta$ and $e_2\zeta = ae_2\alpha+be_2\beta = (ae_2-be_1)\alpha$. We can see that these are linearly independent since the only relation in $N$ is the relation $e_1\alpha = -e_2\beta$, so the Hilbert function of $N$ cannot be $(1,1,0)$.

\end{example}

Our goal now will be to restate Theorem~\ref{thm:HFNndcontainment} as a statement about $\RS(L_{n,d})$ in terms of complexes for which we have a complete characterization of possible rank sequences, namely Koszul complexes. We first introduce some helpful notation: for a nonnegative integer $m$, we will write $E_{(m)} = E/\langle e_{m+1}, \ldots, e_n\rangle$. Note that $E_{(m)} \cong \bbk$ when $m=0$. Furthermore, we will write $I_i$ to refer to the ideal generated by the single degree $-i$ monomial $e_1e_2\cdots e_i$ and we will make the convention that $I_0$ is the unit ideal. We will write $\bfK_{(m)}$ to mean the Koszul complex on $m$ variables for $m\leq n$.

\begin{remark}\label{rmk:twist}
As a $\bbk$-module, $E_{(m)}$ is simply the exterior algebra on $m$ variables, but since we are considering everything over $E$, we have that $\bbL(E_{(m)}) = \bfK_{(m)}[-n+m]$, that is to say, the $i\th$ free module in the complex $\bbL(E_{(m)})$ is the $(i-n+m)\th$ free module in the Koszul complex on $m$ variables.
\end{remark}

\begin{lemma}\label{lem:N1jequalsEnminus1}
For $1\leq j\leq n$,  we have the equality of sets
\[
\HF(\overline{N_{1,d}}) = \HF(E_{(n-1)}). 
\]
\end{lemma}
\begin{proof}
We will show in fact that $\overline{N_{1,d}} \cong E_{(n-1)}$. First observe that $N_{1,d} = \bbk\langle e_1 \rangle / \langle e_1 \rangle = \bbk$. Therefore $\overline{N_{1,d}} = \bbk\otimes \bbk \langle e_2, \ldots, e_n \rangle \cong E_{(n-1)}$.
\end{proof}

\begin{lemma}\label{lem:hfnnd} 
For $n,d\geq 2$, the Hilbert functions of submodules of $N_{n,d}$ are restricted by the following containment:
\[
\HF(N_{n,d}) \subseteq \sum_{i=1}^n\binom{n+d-2-i}{n-i}\HF(\overline{N_{i,1}}).
\]
\end{lemma}
\begin{proof}
Theorem~\ref{thm:HFNndcontainment} gives the containment
\[
\HF(N_{n,d}) \subset \HF(\overline{N_{n-1, d}}) + \HF(N_{n,d-1}).
\]
We can then iterate until we have $\HF(N_{n,d})$ expressed completely in terms of $\HF (\overline{N_{1,j}})$ and $\HF (\overline{N_{i,1}})$ for $2\leq i,j\leq n$. Ultimately, we reduce to
\[
\HF(N_{n,d})\subset \sum_{i=2}^n \alpha_{i,1}\HF(\overline{N_{i,1}}) + \sum_{j=2}^d \alpha_{1,j}\HF(\overline{N_{1,j}})
\]

where $\alpha_{i,j}$ counts the number of times that $N_{i,j}$ appears in the sum. This quantity $\alpha_{i,j}$ is the number of times that $(i,j)$ appears as the result of repeatedly subtracting $(1,0)$ and $(0,1)$ from $(n,d)$, with the caveat that, since $(1,i+1)$ is a base case, we never reach $(1,i)$ by subtracting $(0,1)$ from $(1,i+1)$, and similarly for $(1,j)$. One can thus interpret $\alpha_{i,1}$ as the number of integer lattice paths from $(i,2)$ to $(n,d)$ and $\alpha_{1,j}$ as the number of integer lattice paths from $(2,j)$ to $(n,d)$. The number of such lattice paths from $(i,j)$ to $(n,d)$ is given by $\binom{n-i+d-j}{n-i}$. This gives
\[
\HF(N_{n,d})\subset \sum_{i=2}^n \binom{n+d-2-i}{n-i}\HF(\overline{N_{i,1}}) + \sum_{j=2}^d \binom{n+d-2-j}{n-2}\HF(\overline{N_{1,j}}).
\]
Since $\overline{N_{1,j}} = E_{(n-1)}$ regardless of $j$ by the proof of Lemma~\ref{lem:N1jequalsEnminus1}, we can write the sum
\[
\sum_{j=2}^d \binom{n+d-2-j}{n-2}\HF(\overline{N_{1,j}}) = \left( \sum_{j=2}^d \binom{n+d-2-j}{n-2}\right)\HF(E_{(n-1)}).
\]
We can reindex and convert via the hockey stick identity to see that
\[
\sum_{j=2}^d \binom{n+d-2-j}{n-2} = \sum_{k=n-2}^{n+d-4}\binom{k}{n-2} = \binom{n+d-3}{n-1}
\]
so we have 
\begin{align*}
\HF(N_{n,d})&\subset \binom{n+d-3}{n-1}\HF(E_{(n-1)}) + \sum_{i=2}^n \binom{n+d-2-i}{n-i}\HF(\overline{N_{i,1}})\\
&= \sum_{i=1}^n \binom{n+d-2-i}{n-i}\HF(\overline{N_{i,1}}),
\end{align*}
where the incorporation of the first term into the sum uses the fact that $E_{(n-1)} \cong \overline{N_{1,1}}$

\end{proof}

\begin{lemma}\label{lem:hfni1} 
For $1\leq i\leq n$, we have the containment of sets
\[
\HF(\overline{N_{i,1}}) \subseteq \sum_{j=0}^{i-1} \HF(E_{(n-j-1)}(j)).
\]
\end{lemma}
\begin{proof}
When $d = 1$, the module $N_{i,1}$ is easily computable as $N_{i,1} = \coker \begin{bmatrix} e_1 \cdots e_i \end{bmatrix} = \bbk \langle e_1, \ldots, e_i \rangle / I_i$, and so $\overline{N_{i,1}} = (E_{(i)}/I_i)\otimes \bbk \langle e_{i+1}, \ldots, e_n \rangle = E/I_i$.

Now we can reduce using the short exact sequences of $E/I_i$-modules 
\[
0\to \langle e_i\rangle E/I_i \to E/I_i\to E/(\langle e_i\rangle + I_i)\to 0.
\]
But $\langle e_i\rangle E/I_i \cong E_{(n-1)}/I_{i-1}(1)$ and that $E/(\langle e_i\rangle + I_i) \cong E_{(n-1)}$, so by a similar argument as we have used previously in the proof of Theorem~\ref{thm:HFNndcontainment}, we may now write 
\[
\HF(E/I_m)\subseteq \HF(E_{(n-1)}/I_{m-1}(1)) + \HF(E_{(n-1)}).
\]
Now we can split $\HF(E_{(n-1)}/I_{m-1}(1))$ and proceed inductively to get 
\[
\HF(E/I_i) \subseteq \sum_{j=0}^{i-1}\HF(E_{(n-j-1)}(j)).
\]
\end{proof}

\begin{theorem}\label{thm:RSnumericsEN} For $n,d\geq 2$, the rank sequence of any subcomplex of $L_{n,d}$ can be written as a positive integral sum of rank sequences of Koszul subcomplexes on fewer than $n$ variables. In particular,
\[
\RS(L_{n,d}) \subseteq \sum_{j=0}^{n-1}\binom{n-j+d-2}{d-1}\RS(\bfK_{(n-j-1)}).
\]
\end{theorem}

\begin{proof}
Combining Lemmas~\ref{lem:hfnnd}~and~\ref{lem:hfni1}, we can see that 
\[
\HF(N_{n,d}) \subseteq \sum\limits_{i=1}^{n} \sum\limits_{j=0}^{i-1} \binom{n+d-2-i}{d-2} \HF(E_{(n-j-1)}(j)).
\]
Twisting each side of this equality by $(-n+1)$ gives
\[
\HF(N_{n,d}(-n+1)) \subseteq \sum\limits_{i=1}^{n} \sum\limits_{j=0}^{i-1} \binom{n+d-2-i}{d-2} \HF(E_{(n-j-1)}(-n+j+1)),
\]
which, fed through the functor $\bbL$, yields a containment of sets of rank sequences: 
\[
\RS(L_{n,d}) \subseteq \sum\limits_{i=1}^{n} \sum\limits_{j=0}^{i-1} \binom{n+d-2-i}{d-2} \RS(\bfK_{(n-j-1)}).
\]
We can switch the order of the double sum, reindex, and apply the hockey stick identity once again to conclude the proof:
\[
\begin{aligned}
\sum\limits_{i=1}^n \sum\limits_{j=0}^{i-1} \binom{n+d-2-i}{d-2} \RS(\bfK_{(n-j-1)}) &\subseteq \sum\limits_{j=0}^{n-1} \sum\limits_{i=j+1}^n \binom{n+d-2-i}{d-2} \RS(\bfK_{(n-j-1)}) \\ 
&= \sum\limits_{j=0}^{n-1} \left( \sum\limits_{i=0}^{n-j-1} \binom{d-2+i}{d-2} \right) \RS(\bfK_{(n-j-1)}) \\
&= \sum\limits_{j=0}^{n-1} \binom{n-j+d-2}{d-1} \RS(\bfK_{(n-j-1)}).
\end{aligned}
\]
\end{proof}

\begin{example}
Let $n=4, d=3$. The complex $L_{4,3}$ resolves the ideal of maximal minors of the matrix 
\[
\begin{bmatrix}
x_1 & x_2 & x_3 & x_4 & 0 & 0\\
0 & x_1 & x_2 & x_3 & x_4 & 0\\
0 & 0 & x_1 & x_2 & x_3 & x_4
\end{bmatrix}
\]
and has the form
\[
0 \to S^{10} \to S^{36} \to S^{45} \to S^{20} \to 0.
\]
We can use Theorem ~\ref{thm:RSnumericsEN} to rule out some integer sequences as possible rank sequences for subcomplexes of $L_{4,3}$. The containment in Theorem ~\ref{thm:RSnumericsEN} states that 
\begin{align*}
\RS(L_{4,3}) &\subseteq\sum_{j=0}^3\binom{5-j}{2}\RS(\bfK_{(3-j)})\\
&= 10\RS(\bfK_{(3)}) + 6\RS(\bfK_{(2)}) + 3\RS(\bfK_{(1)}) + \RS(\bfK_{(0)})
\end{align*}

that is, any rank sequence of a subcomplex must be expressible as a sum of 10 rank sequences of subcomplexes of the Koszul complex on 3 variables, 6 rank sequences of subcomplexes of the Koszul complex on 2 variables, 3 rank sequences of subcomplexes of the Koszul complex on 1 variable, and 1 rank sequence of a subcomplex of the Koszul complex on 0 variables. 

If we consider the Koszul complex on 3 variables, Theorem ~\ref{thm:basickoszulranks} tells us that any rank sequence $(r_0, r_1, r_2, r_3)$ of a subcomplex of $\bfK_{(3)}$ with $r_3 = 1$ must have $r_2 = 3$. This means that for a rank sequence $(r_0, r_1, r_2, r_3)$ of a subcomplex of $L_{4,3}$, we must have $r_2\geq 3r_3$. For instance, we may say for certain that the sequence $(10,16,20,8)$ is not a possible rank sequence of a subcomplex of $L_{4,3}$, since $20<3\cdot 8$. In this way, we are able to use our complete characterization of rank sequence of Koszul subcomplexes to eliminate certain potential rank sequences from consideration in the Eagon--Northcott case.
\end{example}

\section{More General Resolutions}

With our previous results for Koszul and Eagon--Northcott complexes resolving powers of the maximal ideal in hand, we turn now to a more general setting. In particular, we are interested in other ideals $I$ that specialize to powers of the maximal ideal in such a way that $S/I$ is still resolved by the Koszul or Eagon--Northcott complex

\subsection{The general Koszul complex}

The Koszul complex can be defined more generally to give a minimal free resolution of a complete intersection. For $f_1, \ldots, f_m \in S$ a regular sequence of homogeneous elements, we replace the differential in definition ~\ref{def:koszul} by $\del_d(e_T) = \sum_{j = 1}^{d} (-1)^{j} f_{i_j} e_{T - i_j}$, adjusting the twists accordingly.

\begin{theorem}\label{thm:generalkoszulranks}
Let $f_1, \ldots, f_m$ be a regular sequence of homogeneous polynomials in $S$. An integer sequence $r = (r_0, \ldots, r_m)$ is in $\RS(\bfK(f_1, \ldots, f_m))$ if and only if it satisfies 
\[
0 \leq r_{i+1} \leq r_i^{(i)} \text{ for } 1 \leq i \leq m-1.
\]
\end{theorem}

\begin{proof}
We will show that $\RS(\bfK(f_1, \ldots, f_m)) = \RS(\bfK(x_1, \ldots, x_m))$, then apply Theorem~\ref{thm:basickoszulranks}.

Let $\bfK$ be the Koszul complex on the variables $x_1, \ldots, x_m$. For $\bfK'$ a general Koszul complex $\bfK(f_1, \ldots, f_m)$ over the ring $S' = \bbk[y_1, \ldots, y_n]$, there is a map $\bfK \to \bfK'$ induced by the map $S\to S'$ sending $x_i$ to $f_i$. 

If $\bfF$ is a subcomplex of $\bfK$, then the image of $\bfF$ under this map is a subcomplex of $\bfK$ with the same rank sequence. So any possible rank sequence of a subcomplex of $\bfK$ must also be possible for a subcomplex of $\bfK'$. 

To see that the possible rank sequences for subcomplexes of $\bfK'$ are \emph{exactly} those that are possible for subcomplexes of $\bfK$, we need to check that given a subcomplex $\bfF'$ of $\bfK'$, the differentials of $\bfF'$ are described by matrices over the subalgebra $R = \bbk[f_1, \ldots, f_m] \subseteq S'$.

For each $i$, we have
\[
\begin{tikzcd}
F'_i \arrow[d] \arrow[r, "\del"] & F'_{i-1} \arrow[d] \\
K'_i \arrow[r, "\del'"] & K'_{i-1}          
\end{tikzcd}
\]

where the vertical maps are given by matrices over $\bbk$. So the differential $\del$ is a matrix over $R$ if and only if $\del'$ is. But entries of $\del'$ are linear in the $f_i$, so they are defined as matrices over $R$. 

Now given a subcomplex $\bfF'$ of $\bfK'$, we need only replace each $F'_i$ by a free $S$-module of the same rank and each $f_i$ in the differential by $x_i$ to obtain a subcomplex $\bfF$ of $\bfK$ with the same rank sequence. 
\end{proof}

\subsection{More general Eagon--Northcott complexes.}
Just as the Koszul complex can be generalized to give a minimal free resolution of a complete intersection, the Eagon--Northcott complex can be generalized to give a minimal free resolution of certain Cohen--Macaulay algebras of the form $S/I$ where $I$ has the maximum possible codimension. We can relate the behavior of subcomplexes of the Eagon--Northcott complex resolving $\frakm^d$ to the behavior of subcomplexes of a general Eagon--Northcott complex as follows. First, we consider a motivating example.

\begin{example}\label{ex:genericEN}
Let $Y = [y_{i,j}]$ be a $p \times q$ generic matrix with $p \leq q$. Then there is a containment of sets 
\[\RS(\bfEN(Y)) \subset \RS(\bfEN(M^{q-p+1,p})).\]
\end{example}

Let $n = pq$, so our matrix is a map $S^q\to S^p$ for $S = \bbk[y_{i,j}] \cong \bbk[x_1, \ldots, x_n]$. This specializes to the matrix 
\[
M^{q-p+1,p} = \begin{bmatrix}
x_1 & x_2 & \cdots &x_{q-p+1} & 0 &\cdots & 0\\
0 & x_1 & \cdots & x_{q-p} & x_{q-p+1} & \cdots & 0\\
\vdots &  &  \ddots & & \ddots & \ddots & \vdots\\
0 & \cdots & 0 & x_1 & \cdots & x_{q-p} & x_{q-p+1}
\end{bmatrix}
\]

under a map that we will call $\varphi$. The maximal minors of this matrix define the ideal $(x_1, \ldots, x_{q-p+1})^p$. 

The map $\varphi$ gives us a map of complexes $\bfEN(Y) \to \bfEN(M^{q-p+1,p})$. What is more, 
under $\varphi$ any subcomplex of $\bfEN(Y)$ gives a subcomplex of $\bfEN(M^{q-p+1,p})$. 

This gives us a containment
\begin{align}\label{ineq:specialize}
\RS(\bfEN(Y)) \subset \RS(\bfEN(M^{q-p+1,p})).
\end{align}

Note that the generic nature of $Y$ had no bearing on the argument in Example~\ref{ex:genericEN}, so a more general statement relating general Eagon--Northcott complexes to the complex $\bfEN(M^{n,d})$ holds by the same reasoning. 

\begin{theorem}\label{thm:generalENranks}
Let $Z$ be a $p \times q$ matrix whose maximal minors define an ideal $I$ whose codimension is $q-p+1$ and where $S/I$ is Cohen--Macaulay. Then there is a containment of sets 
\[\RS(\bfEN(Z)) \subset \RS(\bfEN(M^{q-p+1,p})).\]
\end{theorem}
\begin{proof}

With the hypotheses above, $\bfEN(Z)$ gives a minimal free resolution of $S/I$. Furthermore, the artinian reduction of $S/I$ is isomorphic to $S'/\frakm^p$ for a polynomial ring $S' \cong \bbk[x_1, \ldots, x_{q-p+1}]$. This specialization takes any subcomplex of $\bfEN(Z)$ to a subcomplex of $\bfEN(M^{q-p+1,p})$, so (\ref{ineq:specialize}) holds for $\bfEN(Z)$ as it does in Example~\ref{ex:genericEN}.

\end{proof}

While this theorem relates rank sequences of subcomplexes of the entire complexes $\bfEN(Z)$ and $\bfEN(M^{q-p+1,p})$ rather than their degree $d$ strands, Remark~\ref{rmk:linearstrand} tells us that our restrictions on $\RS(L_{q-p+1,p})$, together with the above theorem, still give us valuable information about $\RS(\bfEN(Z))$. It should be noted, however, that the result in Theorem~\ref{thm:generalENranks} is a strict containment, as demonstrated in the following example.

\begin{example}\label{ex:specialize}
Let $S = \bbk[x,y,z,w]$. Consider the Eagon--Northcott complex on the matrix $\begin{bmatrix}
0 & y & z\\
y & z & 0
\end{bmatrix}$, which resolves the square of the maximal ideal in the subalgebra $\bbk[y,z]\subset S$. 
This is a specialization of the Eagon--Northcott complex on 
$\begin{bmatrix}
x & y & z\\
y & z & w
\end{bmatrix}$
obtained via the map $\varphi: S\to S$ defined by $x,w\mapsto 0$ and $y,z\mapsto y,z$, so we have the following map of complexes: 
\[
\xymatrix{
\bfF':  0 \ar[r] & S(-3)^2 \ar[rr]^{\begin{bsmallmatrix}-z & -w\\ y & z\\ x & -y \end{bsmallmatrix}} \ar[d]_{\varphi^*} & & S(-2)^3 \ar[rrrr]^{\begin{bsmallmatrix} -y^2+xz & -yz+xw & -z^2+yw\end{bsmallmatrix}} \ar[d]_{\varphi^*} & & & & S^1 \ar[d]_{\varphi^*} \\
\bfF: 0 \ar[r] & S(-3)^2 \ar[rr]_{\begin{bsmallmatrix}-z & 0\\ y & z\\ 0 & -y \end{bsmallmatrix}}                       & & S(-2)^3 \ar[rrrr]_{\begin{bsmallmatrix} -y^2 & -yz & -z^2\end{bsmallmatrix}}                            & & & & S^1 
}
\]
Consider the following subcomplex of $\bfF$:
\[
\bfG:  0 \longrightarrow  S(-3) \xrightarrow{\begin{bsmallmatrix} -z\\ y\end{bsmallmatrix}}  S(-2)^2 \xrightarrow{\begin{bsmallmatrix} -y^2 & -yz \end{bsmallmatrix}}  S.
\]

which has $\rs(\bfG) = (1,2,1)$. The subcomplex $\bfG$ is realized as the image of
\[
\bfG': 0 \longrightarrow  S(-3) \xrightarrow{\begin{bsmallmatrix} -z\\ y\end{bsmallmatrix}}  S(-2)^2 \xrightarrow{\begin{bsmallmatrix} -y^2 + xz & -yz +xw \end{bsmallmatrix}}  S
\]
under $\varphi^*$. However, one can check that $\bfG'$ is not a subcomplex of $\bfF'$. Moreover, a straightforward linear algebra computation confirms that there is no subcomplex of $\bfF'$ with rank sequence $(1,2,1)$, so $\RS(\bfF)\subsetneq\RS(\bfG)$.
\end{example}

\bibliographystyle{alpha}
\bibliography{bibliography}

\end{document}